\RequirePackage{amsmath}
\documentclass[runningheads]{llncs}
\usepackage[T1]{fontenc}
\newcommand{\repeatthanks}{\textsuperscript{\thefootnote}}

\usepackage{graphicx}
\usepackage{hyperref}   
\usepackage{url}        
\usepackage{booktabs}   
\usepackage{amsfonts}   
\usepackage{nicefrac}   
\usepackage{microtype}  

\usepackage{float}
\usepackage{subcaption}

\usepackage{algorithm}
\usepackage{algpseudocode}
\usepackage{algorithmicx}
\algrenewcommand\algorithmicrequire{\textbf{Input:}}
\algrenewcommand\algorithmicensure{\textbf{Output:}}

\newtheorem{innercustomgeneric}{\customgenericname}
\providecommand{\customgenericname}{}
\newcommand{\newcustomtheorem}[2]{%
  \newenvironment{#1}[1]
  {%
   \renewcommand\customgenericname{#2}%
   \renewcommand\theinnercustomgeneric{##1}%
   \innercustomgeneric
  }
  {\endinnercustomgeneric}
}
\newcustomtheorem{customlemma}{Lemma}
\newcustomtheorem{customassumption}{Assumption}
\newcustomtheorem{customtheorem}{Theorem}

\usepackage{color}
\begin{document}
\title{Algorithms for Euclidean-regularised Optimal Transport\thanks{
The research was supported by Russian Science Foundation (project No. 23-11-00229), \url{https://rscf.ru/en/project/23-11-00229/}, and by the grant of support for leading scientific schools NSh775.2022.1.1.}}
%
%
\author{Dmitry A. Pasechnyuk\thanks{Equal contribution}\inst{1,2,3,4}\orcidID{0000-0002-1208-1659} \and
Michael Persiianov\repeatthanks\inst{2, 4}\orcidID{0009-0008-7059-1244}\and
Pavel Dvurechensky\inst{5}\orcidID{0000-0003-1201-2343} \and
Alexander Gasnikov\inst{2,4,6}\orcidID{0000-0002-7386-039X}}
\authorrunning{D.A. Pasechnyuk, M. Persiianov et al.}
%
\institute{Mohamed bin Zayed University of Artificial Intelligence, Abu Dhabi, UAE \email{dmitry.vilensky@mbzuai.ac.ae} \and
Moscow Institute of Physics and Technology, Dolgoprudny, Russia
\email{persiianov.mi@phystech.edu,gasnikov.av@mipt.ru} \and
ISP RAS Research Center for Trusted Artificial Intelligence, Moscow, Russia \and 
Institute for Information Transmission Problems RAS, Moscow, Russia
\and
Weierstrass Institute for Applied Analysis and Stochastics, Berlin,
Germany \email{pavel.dvurechensky@wias-berlin.de}\and
Caucasus Mathematic Center of Adygh State University, Maikop, Russia}
\maketitle              
\begin{abstract}
This paper addresses the Optimal Transport problem, which is regularized by the square of Euclidean $\ell_2$-norm. It offers theoretical guarantees regarding the iteration complexities of the Sinkhorn--Knopp algorithm, Accelerated Gradient Descent, Accelerated Alternating Minimisation, and Coordinate Linear Variance Reduction algorithms. Furthermore, the paper compares the practical efficiency of these methods and their counterparts when applied to the entropy-regularized Optimal Transport problem. This comparison is conducted through numerical experiments carried out on the MNIST dataset.
\end{abstract}
\section{Introduction}

Optimal Transport (OT) problem has a long history \cite{kantorovich1942translocation,monge1781memoire}, has been extensively studied \cite{peyre2017computational,villani2009optimal} and piques interest in the modern statistical learning community \cite{arjovsky2017wasserstein,kolouri2017optimal}. This paper focuses on the discrete OT problem statement and the numerical optimisation methods applied to it. Formally, the original problem to solve is:
\begin{equation} \label{eq:original}
    \textstyle\min_{\substack{X\mathbf{1}_m = a\\X^\top\mathbf{1}_n = b\\x_{ij} \geq 0}} \langle C, X \rangle,
\end{equation}
where $a \in \mathcal{S}_n$ and $b \in \mathcal{S}_m$ are the source and destination distributions (measures), the unit simplex $\mathcal{S}_d \equiv \{x \in \mathbb{R}^d_+\;\vert\; \sum_{i=1}^d x_i = 1\}$, $X \in \mathbb{R}^{n\times m}_+$ is a transportation plan such that $x_{ij}$ is the mass to transport from the $i$-th source to the $j$-th destination, and $C \in \mathbb{R}^{n \times m}_+$ is the cost of the transportation matrix.

An algorithm applied to the OT problem must derive an $\varepsilon$-optimal transportation plan, denoted by $X_\varepsilon$ and defined as one that meets the following condition:
\begin{equation*}
    \textstyle\langle C, X_\varepsilon \rangle - \varepsilon \leq \langle C, X^*\rangle \equiv \min_{\substack{X\mathbf{1}_m = a\\X^\top\mathbf{1}_n = b\\x_{ij} \geq 0}} \langle C, X \rangle,
\end{equation*}
and strictly adheres to constraints $X_\varepsilon \mathbf{1}_m = a$, $X^\top_\varepsilon \mathbf{1}_n = b$, and $X_\varepsilon \in \mathbb{R}^{n\times m}_+$. To obtain such a solution, we consider the Euclidean-regularised OT problem:
\begin{equation} \label{eq:primal}
    \textstyle\min_{\substack{X\mathbf{1}_m = a\\X^\top\mathbf{1}_n = b\\x_{ij} \geq 0}} \{f(X) \equiv \langle C, X \rangle + \frac{\gamma}{2} \|X\|_2^2 \},
\end{equation}
where $\|X\|_2^2 \equiv \sum_{i=1,j=1}^{n,m} x_{ij}^2$, and apply convex optimisation methods to solve it. It is noteworthy that if $\gamma \propto \varepsilon$, then the $\varepsilon$-optimum of this optimisation problem is a $(\propto \varepsilon)$-optimal transportation plan for the original problem \eqref{eq:original}. Unlike \eqref{eq:original}, problem statement \eqref{eq:primal} allows one to leverage convex optimisation tools like duality and acceleration.

\textbf{Contribution}. We provide the first arithmetic complexity bounds for Euclidean-regularised OT. The results of this paper are summarised in Table~\ref{table:rates} below. Each cell contains an estimate of the number of arithmetic operations number needed for an Algorithm in the leftmost column to achieve target accuracy $\varepsilon$ for problem \eqref{eq:original} with given $n$, $m$ (we assume without loss of generality that $n > m$), and $C$ in the worst case. Constant factors are omitted, and $\varepsilon$ is assumed to be sufficiently small. The arithmetic complexities for original algorithms applied to entropy-regularised OT \cite{cuturi2013sinkhorn} are known and are presented in the right column. The left column contains the estimates obtained in this paper.
\vspace{-2em}
\begin{table}[ht!]
\caption{Theoretical guarantees on the arithmetic complexity of methods considered in this paper, compared with with those of analogous methods for entropy-regularised problems.}
\label{table:rates}
\centering
{\renewcommand{\arraystretch}{2}
\begin{tabular}{ l|c|c| }
$\#a.o.$ & Euclidean-reg. OT & entropy-reg. OT \\ \hline
Sinkhorn, Alg.~\ref{alg:sinkhorn} & $\displaystyle \frac{n^{7/2} \|C\|_\infty^2}{\varepsilon^2}$, Thm~\ref{thm:sinkhorn-main} & $\displaystyle \frac{n^2 \|C\|_\infty^2 \log{n}}{\varepsilon^2}$, \cite{dvurechensky2018computational} \\ 
APDAGD, Alg.~\ref{alg:apdagd} & $\displaystyle \frac{n^3 \|C\|_\infty} {\varepsilon}$, Thm~\ref{thm:apd-main} & $\displaystyle \frac{n^{5/2} \|C\|_\infty \sqrt{\log{n}}}{\varepsilon}$, \cite{dvurechensky2018computational} \\ 
AAM, Alg.~\ref{alg:aam} & $\displaystyle \frac{n^3 \|C\|_\infty} {\varepsilon}$, Thm~\ref{thm:aam-main} & $\displaystyle \frac{n^{5/2} \|C\|_\infty \sqrt{\log{n}}}{\varepsilon}$, \cite{guminov2021combination} \\
CLVR, Alg.~\ref{alg:vr} (rand.) & $\displaystyle \frac{n^3 \|C\|_\infty} {\varepsilon}$ (on avg.), Thm~\ref{thm:vr-main} & --- \\
\hline
\end{tabular}}
\end{table}
\vspace{-1em}

The organisation of this paper is as follows. Section~\ref{sec:back} provides a short literature review, highlighting the works that underpin the proofs presented in this paper and tracing the history of applying quadratic regularisation in OT. Section~\ref{sec:theory} encompasses all the theoretical results of this paper. Subsections~\ref{sec:sinkhorn}, \ref{sec:aagd}, \ref{sec:aam}, and \ref{sec:vr} delve into the details of the Sinkhorn, Accelerated Gradient, Alternating Minimisation, and Coordinate Linear Variance Reduction algorithms, respectively. Finally, Section~\ref{sec:experiments} contains results of numerical experiments that compare the practical performance of the proposed algorithms and their counterparts applied to entropy-regularised OT.

\section{Background}\label{sec:back}

The Sinkhorn--Knopp algorithm \cite{cuturi2013sinkhorn,sinkhorn1967diagonal} stands out as the most widely-known method to solve the OT problem. The works \cite{altschuler2017near,dvurechensky2018computational} justify its worst-case arithmetic complexity in terms of $\varepsilon$ and $n$. Our analysis of the arithmetic complexity of the Sinkhorn--Knopp algorithm applied to Euclidean-regularised OT draws from the framework outlined in \cite{dvurechensky2018computational} as well. As an alternative to the Sinkhorn--Knopp algorithm, the works \cite{dvurechensky2018computational,lin2019efficient} show that accelerated gradient descent applied to entropy-regularised OT problem improves iteration complexity with respect to $\varepsilon$. on the other hand, acceleration can be applied directly to the Sinkhorn--Knopp algorithm by viewing it as an alternating minimisation procedure, as proposed in \cite{guminov2021combination}. Both approaches yield similar iteration complexities and require only minor adjustments in proofs for applying to Euclidean-regularised OT.

The standard approach for effectively applying convex optimisation methods to the OT is entropy regularisation \cite{cuturi2013sinkhorn}. Recently, there has been a growing interest in Euclidean regularisation \cite{essid2018quadratically,li2020continuous,lorenz2021quadratically}. A practically valuable property of Euclidean-regularised OT is the sparsity of the optimal plan \cite{blondel2018smooth}, which holds significance in various applications, such as image colour transfer. Additionally, algorithms used for Euclidean-regularised OT are anticipated to be more computationally stable and more robust for small regularisation parameter. For instance, the Sinkhorn--Knopp algorithm for entropy-regularised OT requires computing the exponent with $\gamma$ in the denominator. Besides, none of the aforementioned papers that study Euclidean regularisation provide arithmetic complexity estimates for particular algorithms applied to Euclidean-regularised OT.

\section{Theoretical guarantees for various approaches}\label{sec:theory}

\subsection{Common reasoning}

We have two discrete probability measures, $a \in \mathcal{S}_n$ and $b \in \mathcal{S}_m$ from the unit simplex, such that $a^\top \mathbf{1}_n = 1, b^\top \mathbf{1}_m = 1$, along with the cost matrix $C \in \mathbb{R}_+^{n \times m}$. Our objective is to find the transport plan $X \in \mathbb{R}_+^{n \times m}$ determined by optimisation problem \eqref{eq:primal}, which represents the Euclidean-regularised version of the classical problem \eqref{eq:original}.

The problems under consideration are in the generalised linear form and allow for the use of convex duality to eliminate linear constraints. Let us consider the Lagrange saddle-point problem $\max_{\lambda \in \mathbb{R}^n, \mu \in \mathbb{R}^m} \min_{X \in \mathbb{R}_+^{n \times m}} \mathcal{L}(X, \lambda, \mu)$, where the Lagrangian function is defined as follows:
\begin{equation*}
    \textstyle \mathcal{L}(X, \lambda, \mu) \equiv \langle C, X \rangle + \frac{\gamma}{2} \|X\|_2^2 + \lambda^\top (X \mathbf{1}_m - a) + \mu^\top (X^\top \mathbf{1}_n - b).
\end{equation*}
The first-order optimality condition for this problem implies
\begin{equation*}
    \textstyle\frac{\partial \mathcal{L}(X, \lambda, \mu)}{\partial x_{ij}} = 0 = c_{ij} + \gamma x_{ij} + \lambda_i + \mu_j,
\end{equation*}
yielding the following closed-form expression for the optimal transport plan $X(\lambda, \mu) = \left[-C - \lambda \mathbf{1}_m^\top - \mathbf{1}_n \mu^\top \right]_+ / \gamma$, given the dual multipliers $\lambda$ and $\mu$, where $[x]_+ \equiv \max\{0, x\}$. Upon substituting $X(\lambda, \mu)$ into the formula for $\mathcal{L}$, we derive the following dual problem:
\begin{equation} \label{eq:dual}
    \textstyle\max_{\lambda \in \mathbb{R}^n, \mu \in \mathbb{R}^m} \{\varphi(\lambda, \mu) \equiv -\frac{1}{2\gamma} \sum_{j=1}^m \|\left[-C_j - \lambda - \mu_j \mathbf{1}_n\right]_+\|_2^2 - \lambda^\top a - \mu^\top b\},
\end{equation}
where $C_j$ is the $j$-th row of matrix $C$.

\subsection{The Sinkhorn--Knopp Algorithm}\label{sec:sinkhorn}

Following the reasoning of \cite{cuturi2013sinkhorn} regarding the justification of the Sinkhorn--Knopp algorithm for the entropy-regularised OT problem, we come to an analogous Sinkhorn--Knopp method for the Euclidean-regularised OT problem.

The first-order optimality conditions for the dual problem \eqref{eq:dual} with respect to $\lambda$ and $\mu$ are, respectively,
\begin{gather} \label{conds}
    \textstyle\begin{cases}
    f_i(\lambda_i) - \gamma a_i = 0,\; i=1,...,n\\    
    g_j(\mu_j) - \gamma b_j = 0,\; j=1,...,m,
    \end{cases}\\
    \nonumber\textstyle f_i(\lambda) = \sum_{j=1}^m \left[-c_{ij} - \lambda - \mu_j\right]_+,\quad g_j(\mu) = \sum_{i=1}^n \left[-c_{ij} - \lambda_i - \mu\right]_+.
\end{gather}

Let us denote the $i$-th order statistic of the elements of the vector $x$ as $x_{(i)}$, and choose $l$ as the largest index $j$ such that $f_i(-(C^\top_i + \mu)_{(j)}) \leq \gamma a_i$, and $k$ as the largest index $i$ such that $g_j(-(C_j + \lambda)_{(i)}) \leq \gamma b_j$), respectively \cite{lorenz2021quadratically}. Then, by holding $\mu$ and $\lambda$ constant, the explicit solutions of \eqref{conds} are
\begin{equation} \label{eq:altern}
    \begin{cases}
    \lambda_i = -\left(\gamma a_i + \sum_{j=1}^l (C^\top_i + \mu)_{(j)}\right)/\;l,\; i=1,...,n,\\
    \mu_j = -\left(\gamma b_j + \sum_{i=1}^k (C_j + \lambda)_{(i)}\right)/\;k,\; j=1,...,m.
    \end{cases}
\end{equation}

The alternating updates of $\lambda$ and $\mu$ according to the formulas above yield the Sinkhorn--Knopp algorithm applied to Euclidean-regularised OT. Its pseudocode is listed in Algorithm~\ref{alg:sinkhorn}. The following proposition estimates the algorithmic complexity of each iteration of Algorithm~\ref{alg:sinkhorn}.
\begin{proposition}
One iteration of Algorithm~\ref{alg:sinkhorn} requires $\mathcal{O}((n + m)^2)$ amortised arithmetic operations per iteration (only +, -, * and $\leq$; $\mathcal{O}(n + m)$ /; no built-in functions calculations).
\end{proposition}

\begin{algorithm}[ht]
    \caption{Euclidean Sinkhorn--Knopp}
    \label{alg:sinkhorn}
    \begin{algorithmic}[1]
        
        \Require $a, b, C, \gamma, \varepsilon, K$
        
        \For {$k = 0, 1, \dots, K$}
            \If {$k$ is even} 
                \State Iterate over $(C^\top_i + \mu)$ and choose $l$
                \State $\lambda_i = -\left(\gamma a_i + \sum_{j=1}^l (C^\top_i + \mu)_{(j)}\right)/\;l,\; i=1,...,n$
            \Else
                \State Iterate over $(C_j + \lambda)$ and choose $k$
                \State $\mu_j = -\left(\gamma b_j + \sum_{i=1}^k (C_j + \lambda)_{(i)}\right)/\;k,\; j=1,...,m$
            \EndIf
            
            \State $x_{ij} := [-c_{ij} - \lambda_i - \mu_j]_+ / \gamma,\; i=1,...,n,\; j=1,...,m$
            \If {$\sum_{i=1}^n |\sum_{j=1}^m x_{ij} - a_i| + \sum_{j=1}^m |\sum_{i=1}^n x_{ij} - b_j| \leq \varepsilon$} 
                \State \textbf{break} 
            \EndIf
        \EndFor
    \end{algorithmic}
\end{algorithm}

Following Lemmas~\ref{lem:radius}, \ref{lem:disttoopt} and Theorem~\ref{thm:sinkhorn_it} correspond to Lemmas 1, 2 and Theorem 1 from \cite{dvurechensky2018computational}, but the proofs are significantly different from that of their analogues due to the use of specific properties of Euclidean regularisation.

\begin{lemma} \label{lem:radius}
For $R = \|C\|_\infty + \frac{\gamma}{\min\{n, m\}} (1 - {\textstyle \max_{\substack{i=1,...,n\\j=1,...,m}}} \{a_i, b_j\})$, it holds that
\begin{align*}
    \textstyle\max_{j=1,...,m} \mu_j - \min_{j=1,...,m} \mu_j \leq R,\quad &\textstyle\max_{i=1,...,n} \lambda_i - \min_{i=1,...,n} \lambda_i \leq R,\\
    \textstyle\max_{j=1,...,m} \mu^*_j - \min_{j=1,...,m} \mu^*_j \leq R,\quad &\textstyle\max_{i=1,...,n} \lambda^*_i - \min_{i=1,...,n} \lambda^*_i \leq R.
\end{align*}
\end{lemma}
\begin{proof}
Firstly, thanks to the form of updates \eqref{eq:altern}, we can guarantee the non-positivity of dual variables. Indeed, initial values of $\mu$ and $\lambda$ are zero, so non-positive. Then, for all $j=1,...,m$,
\[
    \textstyle\frac{n-1}{\gamma} \mu_j + b_j = \frac{1}{\gamma} \sum_{i=1}^n (-c_{ij} - \lambda_i - \mu_j) \leq \frac{1}{\gamma} \sum_{i=1}^n [-c_{ij} - \lambda_i - \mu_j]_+ = X^\top \mathbf{1}_n = b_j,
\]
that implies $\mu_j \leq 0$. Similarly, one can prove $\lambda_i \leq 0$ for all $i=1,...,n$.

Further, let's relate dual variables with corresponding marginal distributions of $X$. Here we consider only $\mu$, assuming that we just updated it. Similar reasoning can be applied to just updated $\lambda$ as well, that gives the right column of statements from Lemma. 
\begin{align*}
    \textstyle-\mu_j - \|C\|_\infty - \frac{1}{n} \mathbf{1}_n^\top \lambda &\textstyle\leq \frac{\gamma}{n} [X^\top \mathbf{1}_n]_i = \frac{\gamma}{n} b_j \leq \frac{\gamma}{n}\\
    \textstyle-\mu_j -  \frac{1}{n} \mathbf{1}_n^\top \lambda &\textstyle\geq \frac{\gamma}{n} [X^\top \mathbf{1}_n]_i = \frac{\gamma}{n} b_j,\quad \forall j=1,...,m.
\end{align*}
This implies
\[
    \textstyle\mu_j \geq -\|C\|_\infty - \frac{1}{n} (\mathbf{1}_n^\top \lambda + \gamma), \quad \mu_j \leq -\frac{1}{n} (\mathbf{1}_n^\top \lambda + \gamma b_j),\quad \forall j=1,...,m.
\]
Finally,
\begin{align*}
    \textstyle\max_{j=1,...,m} \mu_j - \min_{j=1,...,m} \mu_j &\textstyle\leq -\frac{1}{n} \left(\mathbf{1}_n^\top \lambda + \gamma \max_{j=1,...,m} b_j\right) + \|C\|_\infty + \frac{1}{n} \left(\mathbf{1}_n^\top \lambda + \gamma\right)\\
    &\textstyle= \|C\|_\infty + \frac{\gamma}{n} \left(1 - \max_{j=1,...,m} b_j\right).
\end{align*}
Reasoning for $\mu^*$ and $\lambda^*$ is similar, since the gradient of objective in \eqref{eq:dual} vanishes, so $X^\top \mathbf{1}_n = b$ and $X \mathbf{1}_m = a$, correspondingly.
\end{proof}

\begin{lemma} \label{lem:disttoopt} For $\lambda$, $\mu$, and $X$ taken from each iteration of Algorithm~\ref{alg:sinkhorn} it holds that 
\begin{equation*}
    \varphi(\lambda^*, \mu^*) - \varphi(\lambda, \mu) \leq 4R \sqrt{n + m}(\|X \mathbf{1}_m - a\|_2 + \|X^\top \mathbf{1}_n - b\|_2).
\end{equation*}
\end{lemma}

\begin{proof}
Due to concavity of $\varphi$, we have
\begin{align*}
    \varphi(\lambda^*, \mu^*) \leq \varphi(\lambda, \mu) + \langle \nabla \varphi(\lambda, \mu), (\lambda^*, \mu^*) - (\lambda, \mu) \rangle.
\end{align*}
Then, by H\"older inequality and Lemma~\ref{lem:radius},
\begin{align*}
    &\varphi(\lambda^*, \mu^*) - \varphi(\lambda, \mu) \leq \sqrt{n + m} \|\nabla \varphi(\lambda, \mu)\|_2 \|(\lambda^*, \mu^*) - (\lambda, \mu)\|_\infty\\
    &\qquad\leq 4R\sqrt{n + m} \|\nabla \varphi(\lambda, \mu)\|_2 \leq 4R \sqrt{n + m} (\|X \mathbf{1}_m - a\|_2 + \|X^\top \mathbf{1}_n - b\|_2).
\end{align*}
\end{proof}

\begin{theorem} \label{thm:sinkhorn_it} To obtain $\varepsilon$ solution of problem \eqref{eq:primal}, its sufficient to perform $2 + \frac{8 \max\{n, m\}^{3/2} R}{\gamma \varepsilon}$ iterations of Algorithm~\ref{alg:sinkhorn}.
\end{theorem}

\begin{proof}
Below, $\lambda_+$ and $\mu_+$ will denote values of $\lambda$ and $\mu$ after the current iteration, and $\lambda_{+k}$ and $\mu_{+k}$ denote values of $\lambda$ and $\mu$ after $k$ iterations. Let current update relate to $\lambda$. Denoting $S = -C - \mathbf{1}_n \mu^\top - \lambda \mathbf{1}_m^\top$ and $\delta = \lambda - \lambda_+$, we have
\begin{align*}
    \textstyle\varphi(\lambda_+, \mu_+) - \varphi(\lambda, \mu) &\textstyle= \frac{1}{2 \gamma} \sum_{i,j=0,0}^{n,m} (\max\{0, S_{ij} + \delta_i \}^2 - \max\{0, S_{ij}\}^2) + \delta^\top a\\
    &\textstyle\geq \frac{1}{2 \gamma} \sum_{S_{ij} > 0, \delta_i < 0}(\max\{0, S_{ij} + \delta_i \}^2 - S_{ij}^2) + \delta^\top a\\
    &\textstyle\geq \delta^\top (a + [\delta]_- - 2 \gamma X \mathbf{1}_m) \geq \|\delta\|_2^2 + \delta^\top (a - 2\gamma X \mathbf{1}_m)\\
    &\textstyle\geq \delta^\top (a - X \mathbf{1}_m) \geq \frac{\gamma}{n} \|a - X \mathbf{1}_m\|_2^2,
\end{align*}
due to $\lambda_i - [\lambda_+]_i = \frac{\gamma}{l} a_i - \frac{1}{l} \sum_{j=1}^l (-C^\top_i - \mu - \lambda_i)_{(j)} \geq \frac{\gamma}{l} a - \frac{\gamma}{l} X \mathbf{1}_m$ and for small enough $\gamma$. Then, by Lemma~\ref{lem:disttoopt}, we have
\begin{equation} \label{eq:stepprogress}
    \textstyle \varphi(\lambda_+, \mu_+) - \varphi(\lambda, \mu) \geq \max \left\{\frac{\gamma}{16 n^2} \frac{[\varphi(\lambda^*, \mu^*) - \varphi(\lambda, \mu)]^2}{R^2}, \frac{\gamma}{n} \varepsilon^2 \right\},
\end{equation}
which implies, similarly to \S 2.1.5 from \cite{nesterov2003introductory}, that
\begin{gather}
    \label{eq:kboundcase1}\textstyle  k \leq 1 + \frac{16 n^2 R^2}{\gamma} \frac{1}{[\varphi(\lambda^*, \mu^*) - \varphi(\lambda_+, \mu_+)]} - \frac{16 n^2 R^2}{\gamma} \frac{1}{[\varphi(\lambda^*, \mu^*) - \varphi(\lambda, \mu)]}.
\end{gather}
In the other case of \eqref{eq:stepprogress}, we have
\begin{equation} \label{eq:kboundcase2}
    \textstyle[\varphi(\lambda^*, \mu^*) - \varphi(\lambda_{+k}, \mu_{+k})] \leq [\varphi(\lambda^*, \mu^*) - \varphi(\lambda, \mu)] - \frac{k \gamma \varepsilon^2}{n}.
\end{equation}
To combine bounds on $k$ from \eqref{eq:kboundcase1} and \eqref{eq:kboundcase2}, we take minimum of their sum over all options for current objective function value
\begin{align*}
    \textstyle k &\textstyle \leq \min_{0 \leq s \leq [\varphi(\lambda^*, \mu^*) - \varphi(\lambda, \mu)]} \left\{2 + \frac{16 n^2 R^2}{\gamma s} - \frac{16 n^2 R^2}{\gamma} \frac{1}{[\varphi(\lambda^*, \mu^*) - \varphi(\lambda, \mu)]} + \frac{s n}{\gamma \varepsilon^2}\right\}\\
    &\textstyle = \begin{cases}
      2 + \frac{n}{\gamma} (\frac{8 \sqrt{n} R}{\varepsilon} - \frac{16 n R^2}{[\varphi(\lambda^*, \mu^*) - \varphi(\lambda, \mu)]}) & [\varphi(\lambda^*, \mu^*) - \varphi(\lambda, \mu)] \geq 4 \varepsilon  \sqrt{n} R^2,\\
      2 + \frac{n}{\gamma} \frac{[\varphi(\lambda^*, \mu^*) - \varphi(\lambda, \mu)]}{\varepsilon^2} & [\varphi(\lambda^*, \mu^*) - \varphi(\lambda, \mu)] < 4 \varepsilon  \sqrt{n} R^2,
    \end{cases}
\end{align*}
which implies the statement of Theorem.
\end{proof}

We have not set $R$ and $\gamma$ in the bound above. By Lemma~\ref{lem:radius}, $R \leq \|C\|_\infty + \frac{\gamma}{n}$, so $k \leq 2 + \frac{8 n^{3/2} \|C\|_\infty}{\gamma \varepsilon} + \frac{8 n^{1/2}}{\varepsilon}$, and one can take $\gamma = \varepsilon / 2$, such that solving regularised problem with accuracy $\varepsilon / 4$ will give $(\varepsilon / 2)$-solution of original problem. 
Besides, by Lemma 7 from \cite{altschuler2017near} we have
\[
    \textstyle\langle C, X \rangle \leq \langle C, X^*\rangle + \frac{\gamma}{2} \|X\|_2^2 + 2 (\|a - X \mathbf{1}_m\|_1 + \|b - X^\top \mathbf{1}_n\|_1) \|C\|_\infty,
\]
so one should set target accuracy to $\varepsilon / (4 \|C\|_\infty)$. 
This proves the following result.
\begin{theorem}\label{thm:sinkhorn-main} Number of iterations of Algorithm~\ref{alg:sinkhorn}, sufficient for Algorithm~\ref{alg:approx_ot_sinkhorn} to return $\varepsilon$-optimal transport plan $X$ such that $X \mathbf{1}_m = a, X^\top \mathbf{1}_n = b$, is
\[
    \textstyle\mathcal{O}\left(\frac{(n + m)^{3/2} \|C\|_\infty^2}{\varepsilon^2}\right).
\]
\end{theorem}

\vspace{-1em}
\begin{algorithm}[ht]
    \caption{Approximate OT by Algorithm~\ref{alg:sinkhorn}}
    \label{alg:approx_ot_sinkhorn}
    \begin{algorithmic}[1]
        
        \Require $a, b, C, \varepsilon$
        
        
        \State Find $X'$ for given $C, a, b, \gamma = \varepsilon / 2$, with accuracy $\varepsilon / (4 \|C\|_\infty)$ using Algorithm~\ref{alg:sinkhorn}
            
        \State Find projection $X$ of $X'$ onto the feasible set using Algorithm~2 \cite{altschuler2017near}
    \end{algorithmic}
\end{algorithm}
\vspace{-1em}

Note that correction $a' = \left(1-\varepsilon/8\right)(a + \mathbf{1}_n \varepsilon/(n(8-\varepsilon)))$ of target marginal distributions $a$ and $b$, which is required for original Sinkhorn--Knopp algorithm \cite{dvurechensky2018computational},
is not necessary in Algorithms~\ref{alg:approx_ot_sinkhorn} and \ref{alg:approx_ot_apdagd}, since formula for $R$ from Lemma~\ref{lem:radius} makes sense even if $a_i = 0$ and $b_j = 0$ for some $i$ and $j$. 

\subsection{Adaptive Accelerated Gradient Descent} \label{sec:aagd}

To apply accelerated gradient method to the problem \eqref{eq:primal}, let us consider it as problem of convex optimisation with linear constrains:
\begin{equation} \label{eq:primal_gen}
   \textstyle\min_{\substack{A[X] = B\\x_{ij} \geq 0}} f(X),
\end{equation}
where operator $A: \mathbb{R}^{n \times m} \to \mathbb{R}^{n + m}$ is defined by $A[X] = (X\mathbf{1}_m, X^\top \mathbf{1}_n)$, $B = (a, b) \in \mathbb{R}^{n + m}_+$, $f$ is defined in \eqref{eq:primal}, and corresponding dual problem is equivalent to \eqref{eq:dual}. The following theorem gives iteration complexity for primal-dual Algorithm~\ref{alg:apdagd}, which will be further applied to obtain the solution for problem \eqref{eq:primal}. Note, that for given operator $A$ it holds that 
\begin{equation}\label{eq:A}
\textstyle\|A\|_{2,2} \equiv \sup_{\|X\|_2 = 1} \|A [X]\|_2 = \sqrt{n + m}.
\end{equation}

\begin{theorem}[Theorem 3 \cite{dvurechensky2018computational}] \label{thm:apd} Assume that optimal dual multipliers satisfy $\|(\lambda^*, \mu^*)\|_2 \leq R_2$. Then, Algorithm~\ref{alg:apdagd} generates sequence of approximate solutions for primal and dual problems \eqref{eq:primal_gen} and \eqref{eq:dual}, which satisfy
\begin{align*}
    &\textstyle f(X_k) - f(X^*) \leq f(X_k) - \varphi(\lambda_k, \mu_k) \leq \frac{16 \|A\|_{2,2}^2 R^2}{\gamma k^2},\\
    &\textstyle \|A [X_k] - B\|_2 \leq \frac{16 \|A\|_{2,2}^2 R}{\gamma k^2},\quad \|X_k - X^*\|_2 \leq \frac{8 \|A\|_{2,2} R}{\gamma k}.
\end{align*}
\end{theorem}

\begin{algorithm}[ht]
    \caption{Adaptive Primal-Dual Accelerated Gradient Descent}
    \label{alg:apdagd}
    \begin{algorithmic}[1]
        
        \Require $a, b, \varphi(\cdot), \nabla \varphi(\cdot), L_0, \varepsilon, K$

        \State $\beta_0 = 0$
        \State $\lambda_0 = \lambda'_0 = \widetilde{\lambda}_0 = 0$, $\mu_0 = \mu'_0 = \widetilde{\mu}_0 = 0$
            
        \For {$k = 0, 1, \dots, K$}
            \State $i = 0$
            \Repeat
                \State $L_{k+1} = 2^{i - 1} \cdot L_k$, $i = i + 1$
                \State Solve $L_{k+1} \alpha_{k+1}^2 - \alpha_{k+1} + \beta_{k} = 0$ with respect to $\alpha_{k+1}$
                \State $\beta_{k+1} = \beta_k + \alpha_{k+1}$
                \State $\widetilde{\lambda}_{k} = \lambda_k + \frac{\alpha_{k+1}}{\beta_{k+1}} (\lambda'_k - \lambda_k)$, $\widetilde{\mu}_{k} = \mu_k + \frac{\alpha_{k+1}}{\beta_{k+1}} (\mu'_k - \mu_k)$
                \State $\lambda'_{k+1} = \lambda_k + \alpha_{k+1} \nabla_\lambda \varphi (\widetilde{\lambda}_{k}, \widetilde{\mu}_{k})$, $\mu'_{k+1} = \mu_k + \alpha_{k+1} \nabla_\mu \varphi (\widetilde{\lambda}_{k}, \widetilde{\mu}_{k})$
                \State $\lambda_{k+1} = \lambda_k + \frac{\alpha_{k+1}}{\beta_{k+1}} (\lambda'_{k+1} - \lambda_k)$, $\mu_{k+1} = \mu_k + \frac{\alpha_{k+1}}{\beta_{k+1}} (\mu'_{k+1} - \mu_k)$
            \Until $\varphi(\lambda_{k+1}, \mu_{k+1}) \geq \varphi(\widetilde{\lambda}_{k}, \widetilde{\mu}_{k}) + \langle \nabla_\lambda \varphi(\widetilde{\lambda}_{k}, \widetilde{\mu}_{k}), \lambda_{k+1} - \widetilde{\lambda}_{k} \rangle + \langle \nabla_\mu \varphi(\widetilde{\lambda}_{k}, \widetilde{\mu}_{k}), \mu_{k+1} - \widetilde{\mu}_{k} \rangle + \frac{L_{k+1}}{2} (\|\lambda_{k+1} - \widetilde{\lambda}_{k}\|_2^2 + \|\mu_{k+1} - \widetilde{\mu}_{k}\|_2^2)$
            \State $X_{k+1} = X_k + \frac{\alpha_{k+1}}{\beta_{k+1}} (X(\lambda_{k+1}, \mu_{k+1}) - X_k)$
            \If {$f(X_{k+1}) - \varphi(\lambda_{k+1}, \mu_{k+1}) \leq \varepsilon$, $\|X_{k+1} \mathbf{1}_m - a\|_2^2 + \|X_{k+1}^\top \mathbf{1}_n - b\|_2^2 \leq \varepsilon^2$} 
                \State \textbf{break} 
            \EndIf
        \EndFor
    \end{algorithmic}
\end{algorithm}
\vspace{-1em}

Following the proof scheme chosen in \cite{dvurechensky2018computational}, we estimate the error of solution $X$ for the original problem \eqref{eq:original}:
\begin{align*}
    \textstyle\langle C, X\rangle &= \langle C, X^*\rangle + \langle C, X_{\text{reg.}}^* - X^* \rangle + \langle C, X_k - X_{\text{reg.}}^* \rangle + \langle C, X - X_k \rangle\\
    &\textstyle\leq \langle C, X^*\rangle + \langle C, X_{\text{reg.}}^* - X^* \rangle + \langle C, X - X_k\rangle + f(X_k) + \varphi(\lambda_k, \mu_k) + \gamma,
\end{align*}
where $X^*_{\text{reg.}}$ is the exact solution of problem \eqref{eq:primal}. By choosing $\gamma \leq \varepsilon/3$, obtaining $X_k$ such that $f(X_k) - \varphi(\lambda_k, \mu_k) \leq \varepsilon/3$ by Algorithm~\ref{alg:apdagd} and making $\langle C, X - X_k\rangle \leq \varepsilon/3$, we guarantee arbitrarily good approximate solution $X$. Let us consider the latter condition in more details. By Lemma 7 \cite{altschuler2017near} and Theorem~\ref{thm:apd} one has 
\begin{align*}
    \textstyle\langle C, X - X_k\rangle &\textstyle\leq \|C\|_\infty \|X - X_k\|_1 \leq 2 \|C\|_\infty (\|X_k \mathbf{1}_m - a\|_1 + \|X_k^\top \mathbf{1}_n - b\|_1)\\
    &\textstyle\leq_1 2 \sqrt{n + m} \|C\|_\infty \|A [X_k] - B\|_2 \leq \frac{32 (n + m)^{3/2} \|C\|_\infty R}{\gamma k^2}\\
    &\textstyle\leq_2 2 \sqrt{n + m} \|C\|_\infty \|X_k - X_{\text{reg.}}^*\|_2 \leq \frac{16 (n + m) \|C\|_\infty R}{\gamma k}.
\end{align*}
To ensure the latter, it is sufficient to choose $k$ such that
\begin{gather}
    \label{eq:apd-intermediate}\textstyle k = \mathcal{O}\left(\min\left\{\frac{n \|C\|_\infty R}{\varepsilon^2}, \frac{n^{3/4} \sqrt{\|C\|_\infty R}}{\varepsilon}\right\}\right).
\end{gather}
On the other hand, $f(X_k) - \varphi(\lambda_k, \mu_k) \leq \varepsilon/3$ together with Theorem~\ref{thm:apd} imply
\[
\textstyle k = \mathcal{O}\left(\frac{\sqrt{n + m} R}{\varepsilon}\right),
\]
which is majorated by \eqref{eq:apd-intermediate} and does not contribute to iteration complexity. This proves, taking into account \eqref{eq:A}, Lemma~\ref{lem:radius}, and that $R_2 \leq R \sqrt{n + m}$, the following result.
\begin{theorem}\label{thm:apd-main}
Number of iterations of Algorithm~\ref{alg:apdagd}, sufficient for Algorithm~\ref{alg:approx_ot_apdagd} to return $\varepsilon$-optimal transport plan $X$ such that $X \mathbf{1}_m = a, X^\top \mathbf{1}_n = b$, is
\[
\textstyle \mathcal{O}\left(\min\left\{\frac{(n + m)^{3/2} \|C\|_\infty^2}{\varepsilon^2}, \frac{(n + m) \|C\|_\infty} {\varepsilon}\right\}\right).
\]
\end{theorem}
\vspace{-1em}
\begin{algorithm}[ht]
    \caption{Approximate OT by Algorithms~\ref{alg:apdagd},\ref{alg:aam}, or \ref{alg:vr}}
    \label{alg:approx_ot_apdagd}
    \begin{algorithmic}[1]
        
        \Require $a, b, C, \varepsilon$
        
        
        \State Find $X'$ for given $C, a, b, \gamma = \varepsilon / 3$, with accuracy $\varepsilon/3$ using Algorithms~\ref{alg:apdagd}, \ref{alg:aam}, or \ref{alg:vr}
            
        \State Find projection $X$ of $X'$ onto the feasible set using Algorithm~2 \cite{altschuler2017near}
    \end{algorithmic}
\end{algorithm}
\vspace{-1em}
\subsection{Accelerated Alternating Minimisation}\label{sec:aam}

Note that Sinkhorn--Knopp algorithm is based on the simplest alternating optimisation scheme: dual function $\varphi$ is explicitly optimised with respect to $\lambda$ and $\mu$ alternately. Thus, if there is a way to accelerate some alternating optimisation algorithm, similar technique can be applied to Sinkhorn--Knopp algorithm. Moreover, iteration complexity will correspond to that of taken accelerated alternating optimisation method, while the arithmetic complexity of optimisation with respect to one variable will be the same as for Sinkhorn algorithm.

The following theorem gives iteration complexity for general primal-dual alternating minimisation Algorithm~\ref{alg:aam}, which can be used similarly to Algorithm~\ref{alg:apdagd} to obtain the solution for problem \eqref{eq:primal}. Note that $b$, which denotes the number of independent variables blocks in \cite{guminov2021combination}, can be set to $b = 2$ in our case, because $\|\nabla_\lambda \varphi(\lambda, \mu)\|_2 >  \|\nabla_\mu \varphi(\lambda, \mu)\|_2$ implies $\|\nabla_\lambda \varphi(\lambda, \mu)\|_2^2 > \frac{1}{2} \|\nabla \varphi(\lambda, \mu)\|_2^2$. But since dimensionalities of $\lambda$ and $\mu$ are different, one of the variables which has bigger dimensionality will be updated more often a priori. 

\begin{theorem}[Theorem 3 \cite{guminov2021combination} for $b = 2$] \label{thm:am} Assume that optimal dual multipliers satisfy $\|(\lambda^*, \mu^*)\|_2 \leq R_2$. Then, Algorithm~\ref{alg:aam} generates sequence of approximate solutions for primal and dual problems \eqref{eq:primal_gen} and \eqref{eq:dual}, which satisfy
\begin{align*}
    &\textstyle f(X_k) - f(X^*) \leq f(X_k) - \varphi(\lambda_k, \mu_k) \leq \frac{16 \|A\|_{2,2}^2 R^2}{\gamma k^2},\\
    &\textstyle \|A [X_k] - B\|_2 \leq \frac{16 \|A\|_{2,2}^2 R}{\gamma k^2},\quad \|X_k - X^*\|_2 \leq \frac{8 \|A\|_{2,2} R}{\gamma k},
\end{align*}
\end{theorem}

\begin{algorithm}[ht]
    \caption{Primal-Dual Accelerated Alternating Minimisation}
    \label{alg:aam}
    \begin{algorithmic}[1]
        
        \Require $a, b, \varphi(\cdot), \nabla \varphi(\cdot), L_0, \varepsilon, K$

        \State $\beta_0 = 0$, $\lambda_0 = \lambda'_0 = \widetilde{\lambda}_0 = 0$, $\mu_0 = \mu'_0 = \widetilde{\mu}_0 = 0$
            
        \For {$k = 0, 1, \dots, K$}
            \State $i = 0$
            \Repeat
                \State $L_{k+1} = 2^{i - 1} \cdot L_k$, $i = i + 1$
                \State $\alpha_{k+1} = \frac{1}{L_{k+1}} + \sqrt{\frac{1}{4 L_{k+1}^2} + \frac{\alpha_{k} L_k}{L_{k+1}}}$
                \State $\widetilde{\lambda}_{k} = \lambda_k + \frac{1}{\alpha_{k+1} L_{k+1}}(\lambda'_k - \lambda_k)$, $\widetilde{\mu}_{k} = \mu_k + \frac{1}{\alpha_{k+1} L_{k+1}}(\mu'_k - \mu_k)$
                \If{$\|\nabla_\lambda \varphi(\widetilde{\lambda}_{k+1}, \widetilde{\mu}_{k+1})\| > \|\nabla_\mu \varphi(\widetilde{\lambda}_{k+1}, \widetilde{\mu}_{k+1})\|$}
                    \State $\lambda_{k+1} = \arg \max_{\lambda} \varphi(\lambda, \widetilde{\mu}_k)$, $\mu_{k+1} = \mu_k$
                \Else
                    \State $\lambda_{k+1} = \lambda_k$, $\mu_{k+1} = \arg \max_{\mu} \varphi(\widetilde{\lambda}_k, \mu)$
                \EndIf
            \Until $\varphi(\lambda_{k+1}, \mu_{k+1}) \geq \varphi(\widetilde{\lambda}_{k+1}, \widetilde{\mu}_{k+1}) + \frac{\|\nabla \varphi(\widetilde{\lambda}_{k+1}, \widetilde{\mu}_{k+1})\|_2^2}{2 L_{k+1}}$
            \State $\lambda'_{k+1} = \lambda_k + \alpha_{k+1} \nabla_\lambda \varphi (\widetilde{\lambda}_{k}, \widetilde{\mu}_{k})$, $\mu'_{k+1} = \mu_k + \alpha_{k+1} \nabla_\mu \varphi (\widetilde{\lambda}_{k}, \widetilde{\mu}_{k})$
            \State $X_{k+1} = \frac{1}{\alpha_{k+1} L_{k+1}} X(\widetilde{\lambda}_{k}, \widetilde{\mu}_{k}) + \frac{\alpha_k^2 L_k}{\alpha_{k+1}^2 L_{k+1}} X_k$
            \If {$f(X_{k+1}) - \varphi(\lambda_{k+1}, \mu_{k+1}) \leq \varepsilon$, $\|X_{k+1} \mathbf{1}_m - a\|_2^2 + \|X_{k+1}^\top \mathbf{1}_n - b\|_2^2 \leq \varepsilon^2$} 
                    \State \textbf{break} 
                \EndIf
        \EndFor
    \end{algorithmic}
\end{algorithm}

Instead of $\arg \max$ operator taking place in the listing of general Algorithm~\ref{alg:aam} one should use formulas \eqref{eq:altern}. The advantage of this approach consists in simplicity of obtaining the solution for these auxiliary problems. It is expected that while accelerated gradient descent considered before was making one gradient step at each iteration, this algorithm makes optimal step with respect to half of dual variables, so expected progress per iteration is bigger, while the number of iterations is the same up to small $\mathcal{O}(1)$ factor. Using the proof scheme similar to which is provided in Section~\ref{sec:aagd} and the same problem pre- and post-processing Algorithm~\ref{alg:approx_ot_apdagd}, one can guarantee, taking into account \eqref{eq:A} and Lemma~\ref{lem:radius}, that the following result holds.

\begin{theorem} \label{thm:aam-main}
Number of iterations of Algorithm~\ref{alg:aam}, sufficient for Algorithm~\ref{alg:approx_ot_apdagd} to return $\varepsilon$-optimal transport plan $X$ such that $X \mathbf{1}_m = a, X^\top \mathbf{1}_n = b$, is
\[
\textstyle\mathcal{O}\left(\min\left\{\frac{(n + m)^{3/2} \|C\|_\infty^2}{\varepsilon^2}, \frac{(n + m) \|C\|_\infty} {\varepsilon}\right\}\right).
\]
\end{theorem}

\subsection{Coordinate Linear Variance Reduction}\label{sec:vr}

One can also consider problem \eqref{eq:primal} as generalised linear problem with strongly-convex regulariser and sparse constraints. By using the property that dual variables or problem \eqref{eq:dual} are separable into two groups ($\lambda$ and $\mu$), one can apply primal-dual incremental coordinate methods. One of the modern algorithms which is based on dual averaging and has implicit variance reduction effect was proposed in \cite{song2022coordinate}. The following theorem presents simplified form of iteration complexity estimate for Algorithm~\ref{alg:vr} adopted to our particular problem.

\begin{algorithm}[ht]
    \caption{Coordinate Linear Variance Reduction}
    \label{alg:vr}
    \begin{algorithmic}[1]
        
        \Require $a, b, C, \gamma, \varepsilon, K, \alpha$

        \State $A_0 = a_0 = 1/(2\sqrt{n + m})$, $X_0 = \{1/(nm)\}, \lambda_0 = 0, \mu_0 = 0, z_{-1} = 0, q_{-1} = a_0 C$
            
        \For {$k = 0, 1, \dots, K$}
            \State $X_{k+1} = [\alpha X_0 - q_{k-1}]_+ / (\alpha + \gamma A_{k+1})$
            \State Generate uniformly random $\xi_k \in [0, 1]$ 
            \If{$\xi_k < 0.5$}
                \State $\lambda_{k+1} = \lambda_k + 2 \gamma a_k (X_{k+1} \mathbf{1}_m - a)$, $\mu_{k+1} = \mu_k$
            \Else
                \State $\lambda_{k+1} = \lambda_k$, $\mu_{k+1} = \mu_k + 2 \gamma a_k (X_{k+1}^\top \mathbf{1}_n - b)$
            \EndIf
            \State $a_{k+1} = \frac{1}{4}\sqrt{\frac{1 + \gamma A_k / \alpha}{n + m}}$, $A_{k+1} = A_k + a_k$
            \If{$\xi_k < 0.5$}
                \State $z_k = z_{k-1} + (\lambda_k - \lambda_{k-1}) \mathbf{1}_m^\top$
            \Else
                \State $z_k = z_{k-1} + \mathbf{1}_n (\mu_k - \mu_{k-1})^\top$
            \EndIf
            \State $q_k = q_{k-1} + a_{k+1} (z_k + C) + 2 a_k (z_k - z_{k-1})$
            \State $\widetilde{X}_{k+1} = \frac{1}{A_{k+1}} \sum_{i=1}^{k+1} a_i X_i$
            \If {$f(\widetilde{X}_{k+1}) - \varphi(\lambda_{k+1}, \mu_{k+1}) \leq \varepsilon$, $\|\widetilde{X}_{k+1} \mathbf{1}_m - a\|_2^2 + \|\widetilde{X}_{k+1}^\top \mathbf{1}_n - b\|_2^2 \leq \varepsilon^2$} 
                    \State \textbf{break} 
                \EndIf
        \EndFor
    \end{algorithmic}
\end{algorithm}
\vspace{-1em}

\begin{theorem}[Corollary 1 \cite{song2022coordinate} for $b = 2$]\label{thm:vr} Assume that optimal dual multipliers satisfy $\|(\lambda^*, \mu^*)\|_2 \leq R_2$. Then, Algorithm~\ref{alg:vr} generates sequence of approximate solutions for primal and dual problems \eqref{eq:primal_gen} and \eqref{eq:dual}, which satisfy
\begin{gather*}
    \textstyle \mathbb{E}[f(\widetilde{X}_k) - f(X^*)] = \mathcal{O}\left(\frac{\|A\|_{2,2}^2 R^2}{\gamma k^2}\right),\; \mathbb{E}[\|A[\widetilde{X}_k] - B\|_2] = \mathcal{O}\left(\frac{\|A\|_{2,2}^2 R}{\gamma k^2}\right).
\end{gather*}
\end{theorem}

Taking into account \eqref{eq:A} and Lemma~\ref{lem:radius}, using the same reasoning as for Theorem~\ref{thm:apd-main}, one has
\begin{theorem} \label{thm:vr-main}
Number of iterations of Algorithm~\ref{alg:vr}, sufficient for Algorithm~\ref{alg:approx_ot_apdagd} to return expected $\varepsilon$-optimal transport plan $X$ such that $X \mathbf{1}_m = a, X^\top \mathbf{1}_n = b$, is
\[
\textstyle\mathcal{O}\left(\min\left\{\frac{(n + m)^{3/2} \|C\|_\infty^2}{\varepsilon^2}, \frac{(n + m) \|C\|_\infty} {\varepsilon}\right\}\right),
\]
where ``expected $\varepsilon$-optimal'' means that $\mathbb{E}[\langle C, X\rangle] - \varepsilon \leq \langle C, X^*\rangle$.
\end{theorem}

One can see that asymptotic of iteration complexity is the same as that of Algorithms~\ref{alg:apdagd} and \ref{alg:aam}. This allows to use the same pre- and post-processing Algorithm~\ref{alg:approx_ot_apdagd} to apply this algorithm to the OT problem. The advantage of this algorithm is the simplicity of iterations. It is expect that despite the same $\mathcal{O}(nm)$ arithmetic complexity of one iteration, constant of it in practice is significantly smaller than for accelerated methods considered before. 

\section{Numerical experiments} \label{sec:experiments}

All the optimisation algorithms described in previous section are implemented in Python 3 programming language. Reproduction package including source code of algorithms and experiments settings is hosted on GitHub\footnote{Repository is available at \url{https://github.com/MuXauJl11110/Euclidean-Regularised-Optimal-Transport}}. We consider OT problem for the pair of images from MNIST dataset \cite{deng2012mnist}, where distributions are represented by vectorised pixel intensities and cost matrix contains pairwise Euclidean distances between pixels. 
\begin{figure}[ht!]
     \centering
     \begin{subfigure}[b]{0.3\textwidth}
         \centering
         \includegraphics[width=\textwidth]{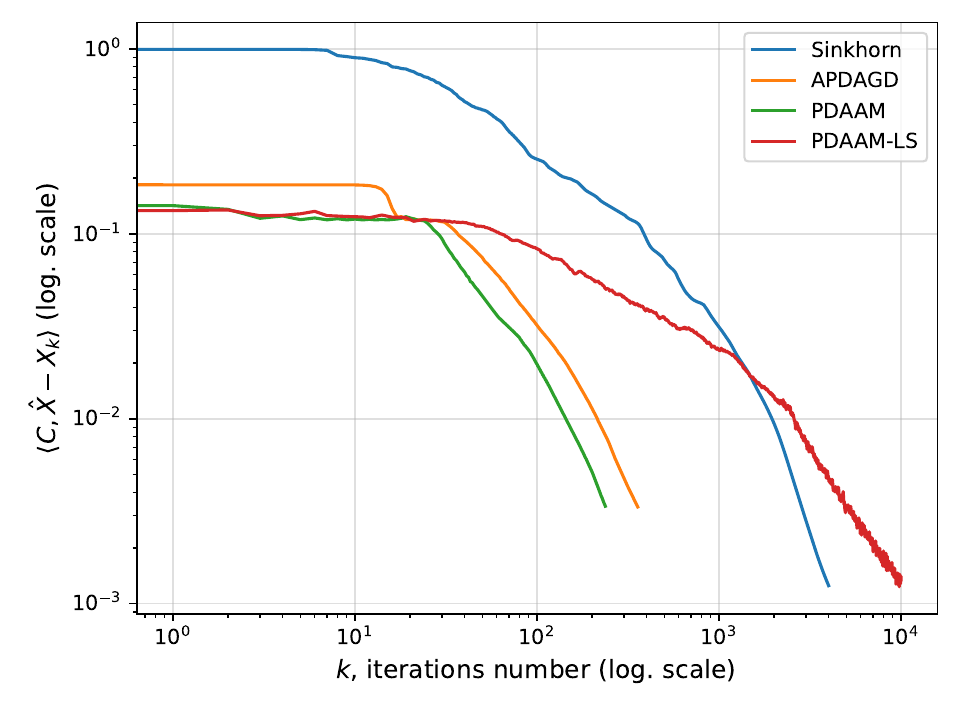}
         \caption{Function value convergence curves, $\varepsilon \propto 10^{-2}$.}
        \label{fig:fun_a}
     \end{subfigure}
     \begin{subfigure}[b]{0.3\textwidth}
         \centering
         \includegraphics[width=\textwidth]{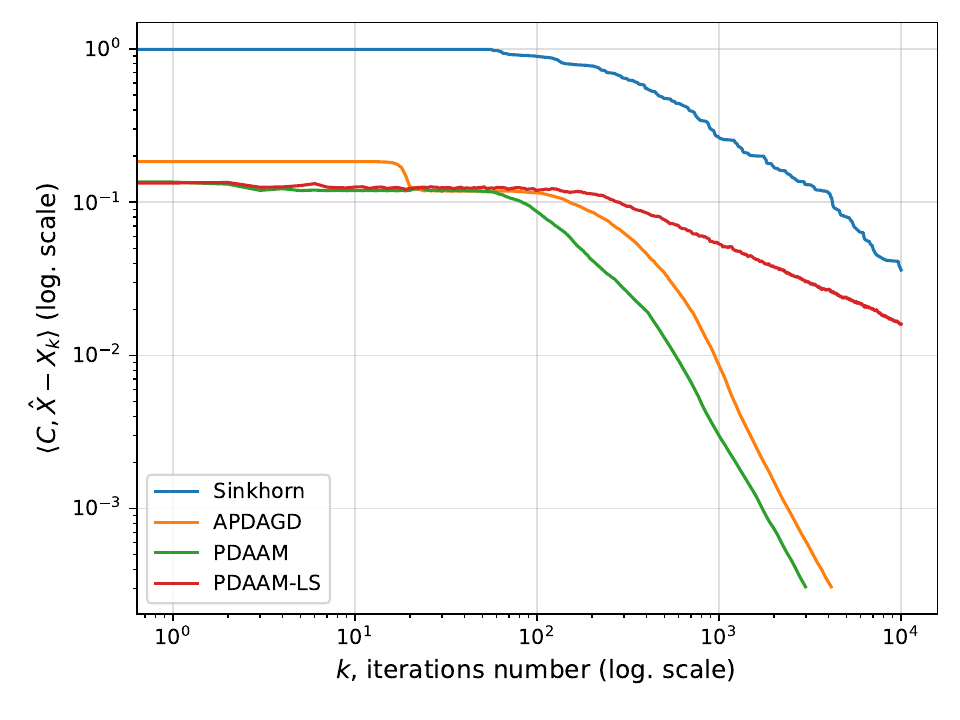}
         \caption{Function value convergence curves, $\varepsilon \propto 10^{-3}$.}
        \label{fig:fun_b}
     \end{subfigure}
     \begin{subfigure}[b]{0.3\textwidth}
         \centering
         \includegraphics[width=\textwidth]{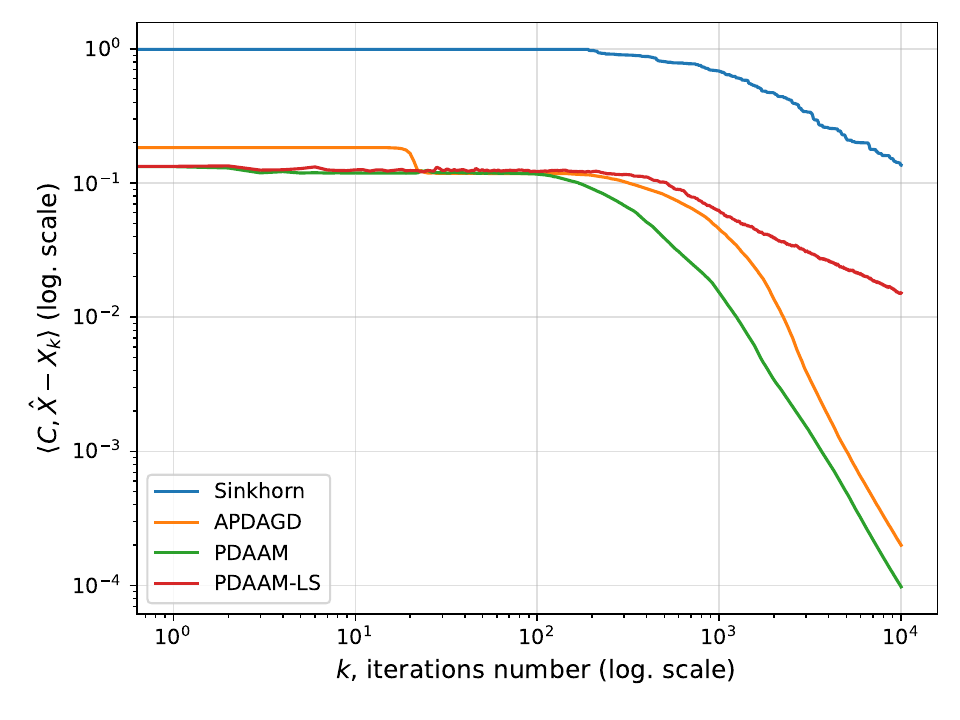}
         \caption{Function value convergence curves, $\varepsilon \propto 10^{-4}$.}
        \label{fig:fun_c}
     \end{subfigure}
     \begin{subfigure}[b]{0.3\textwidth}
         \centering
         \includegraphics[width=\textwidth]{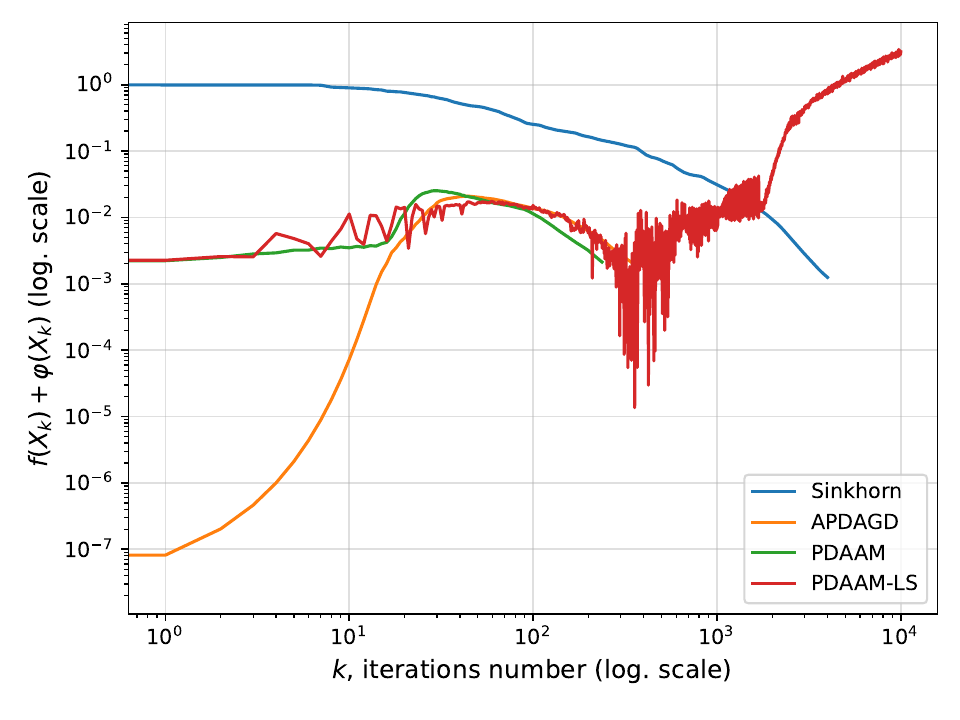}
         \caption{Dual gap convergence curves, $\varepsilon \propto 10^{-2}$.}
        \label{fig:gap_a}
     \end{subfigure}
     \begin{subfigure}[b]{0.3\textwidth}
         \centering
         \includegraphics[width=\textwidth]{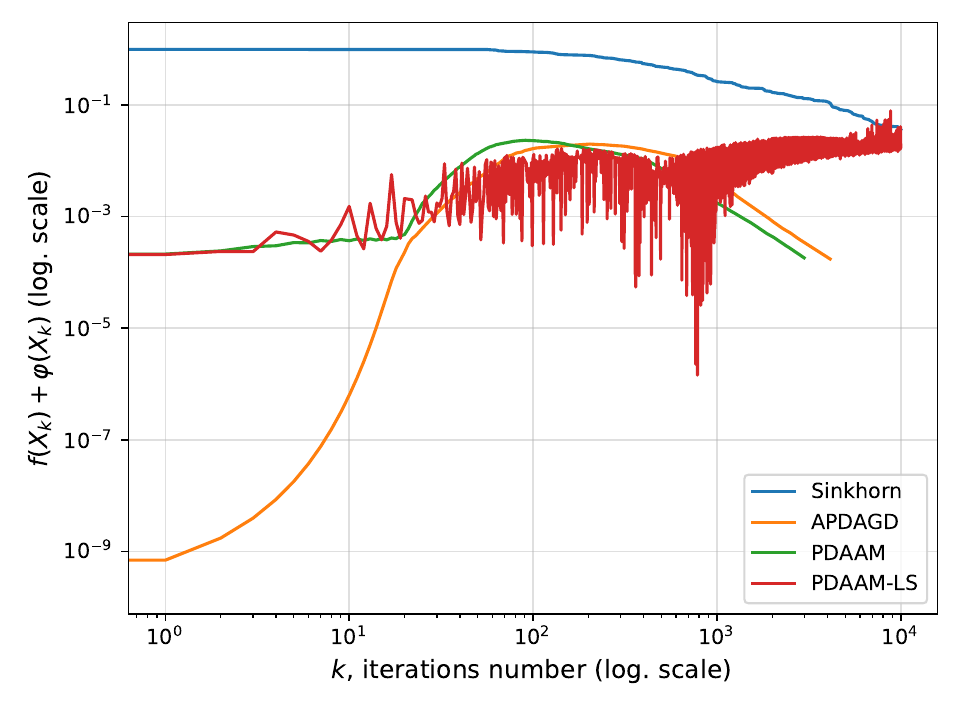}
         \caption{Dual gap convergence curves, $\varepsilon \propto 10^{-3}$.}
        \label{fig:gap_b}
     \end{subfigure}
     \begin{subfigure}[b]{0.3\textwidth}
         \centering
         \includegraphics[width=\textwidth]{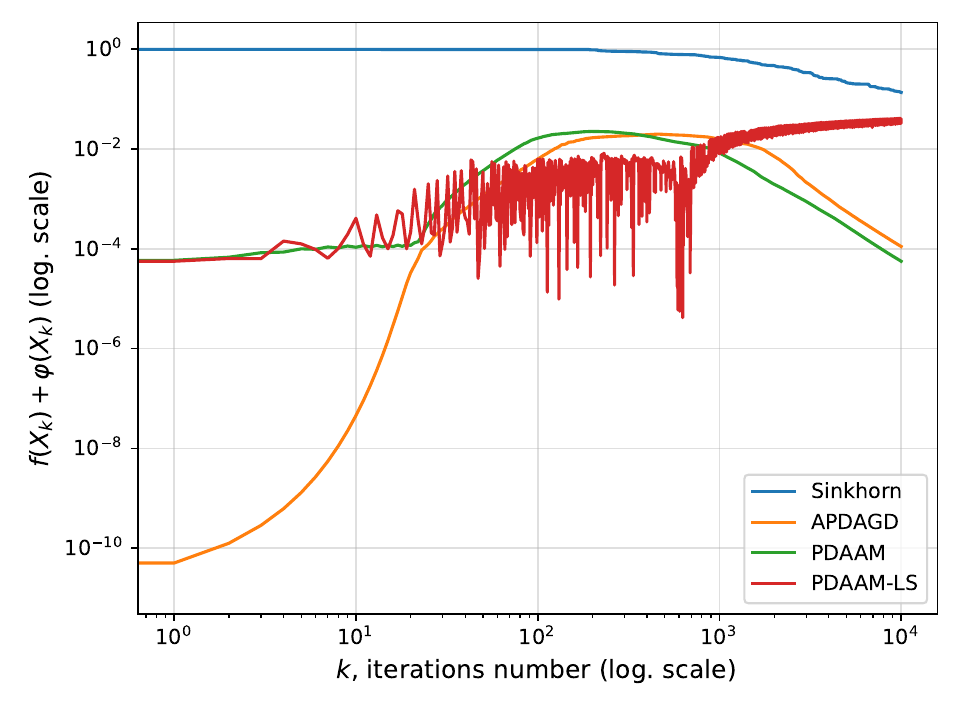}
         \caption{Dual gap convergence curves, $\varepsilon \propto 10^{-4}$.}
        \label{fig:gap_c}
     \end{subfigure}
        \caption{Practical efficiency of Sinkhorn--Knopp, Adaptive Accelerated Gradient, and Accelerated Alternating methods applied to entropy-regularised OT problem on MNIST dataset.}
        \label{fig:entropy}
\end{figure}

Firstly, experiment on comparison of algorithms applied to entropy-regularised OT was carried out. Following algorithms were compared: Sinkhorn--Knopp algorithm (Sinkhorn) \cite{dvurechensky2018computational}, Adaptive Primal-dual Accelerated Gradient Descent (APDAGD) \cite{dvurechensky2018computational}, Primal-dual Accelerated Alternating Minimisation (PDAAM) \cite{guminov2021combination} and its modification which uses one-dimensional optimisation to choose step size (PDAAM-LS). Results of the experiment are shown in Figure~\ref{fig:entropy}. There are presented convergence curves of methods for two progress measures: function value for original problem \eqref{eq:original} and dual gap for problem \eqref{eq:dual}. The range of target accuracy value is $\varepsilon \in \{2 \cdot 10^{-2}, 1.85 \cdot 10^{-3}, 5 \cdot 10^{-4}\}$ (each target accuracy value requires separate experiment, because $\varepsilon$ is a parameters of Algorithms~\ref{alg:approx_ot_sinkhorn} and \ref{alg:approx_ot_apdagd} and affects the convergence from the beginning).

All the plots show that PDAAM is leading algorithm, and performance of APDAGD is competitive with it. On the other hand, Sinkhorn--Knopp algorithm converges slowly, especially for small $\varepsilon$. PDAAM-LS demonstrates unstable behaviour in our experiment. 
\begin{figure}[ht!]
     \centering
     \begin{subfigure}[b]{0.3\textwidth}
         \centering
         \includegraphics[width=\textwidth]{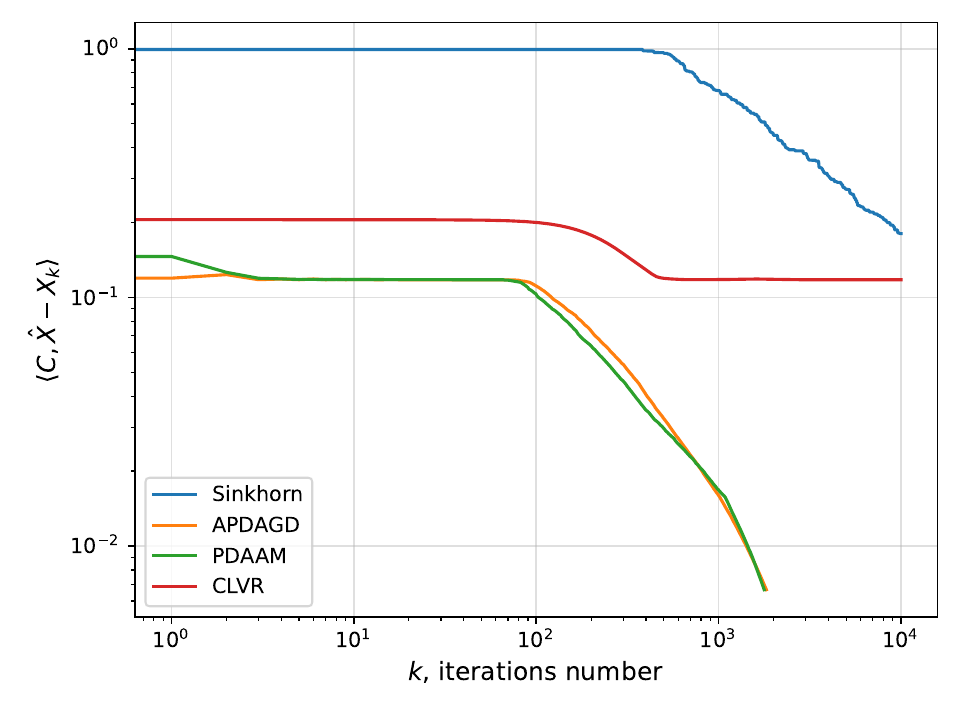}
         \caption{Function value convergence curves, $\varepsilon \propto 10^{-2}$.}
        \label{fig:new_fun_a}
     \end{subfigure}
     \begin{subfigure}[b]{0.3\textwidth}
         \centering
         \includegraphics[width=\textwidth]{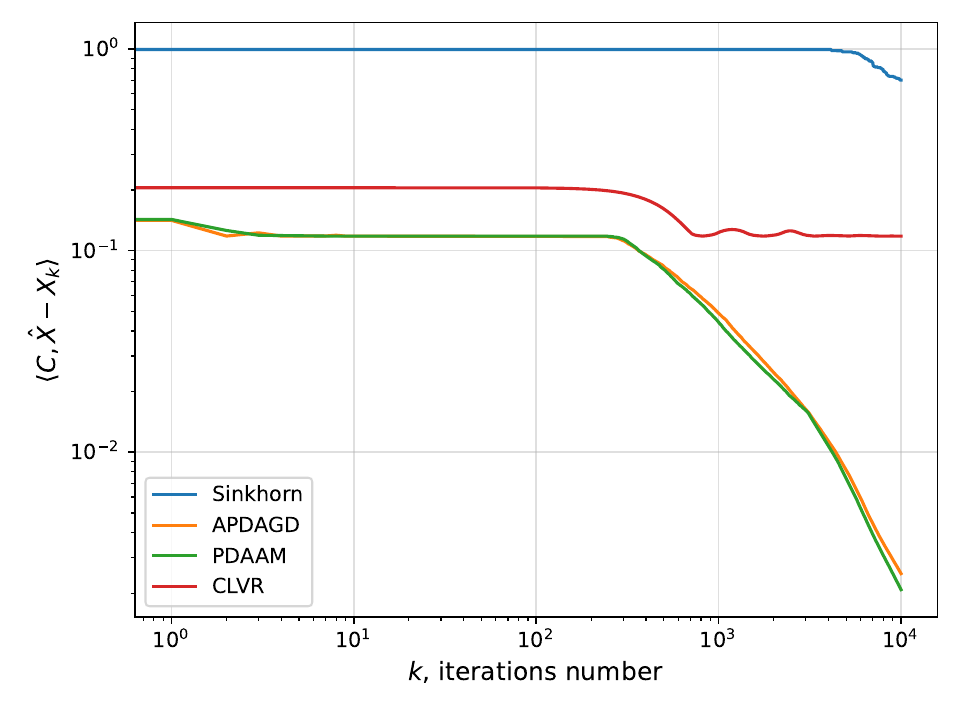}
         \caption{Function value convergence curves, $\varepsilon \propto 10^{-3}$.}
        \label{fig:new_fun_b}
     \end{subfigure}
     \begin{subfigure}[b]{0.3\textwidth}
         \centering
         \includegraphics[width=\textwidth]{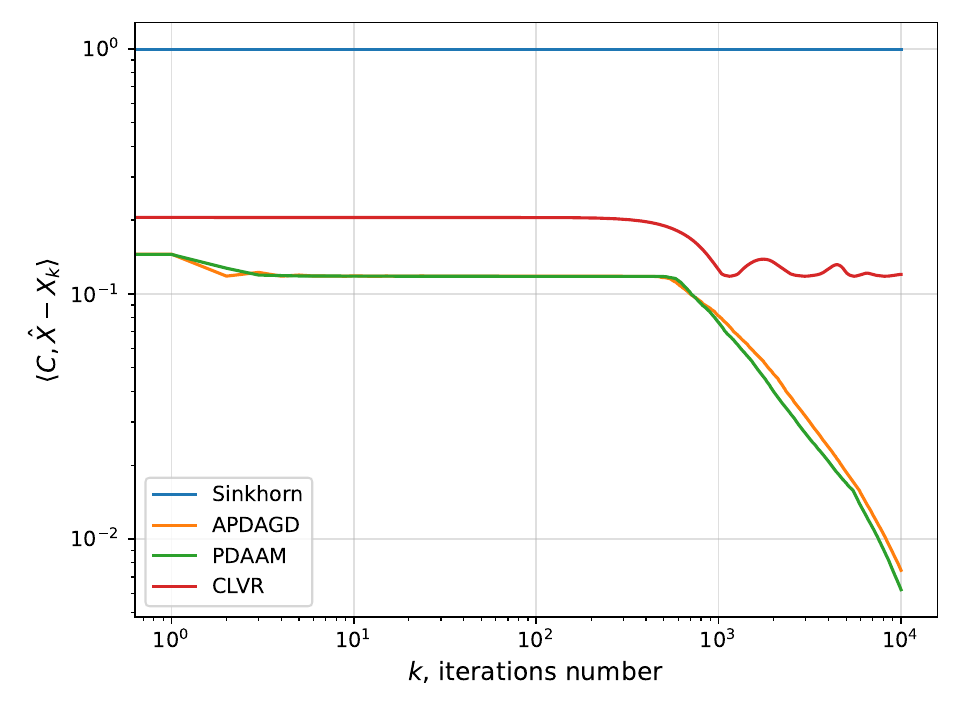}
         \caption{Function value convergence curves, $\varepsilon \propto 10^{-4}$.}
        \label{fig:new_fun_c}
     \end{subfigure}
     \begin{subfigure}[b]{0.3\textwidth}
         \centering
         \includegraphics[width=\textwidth]{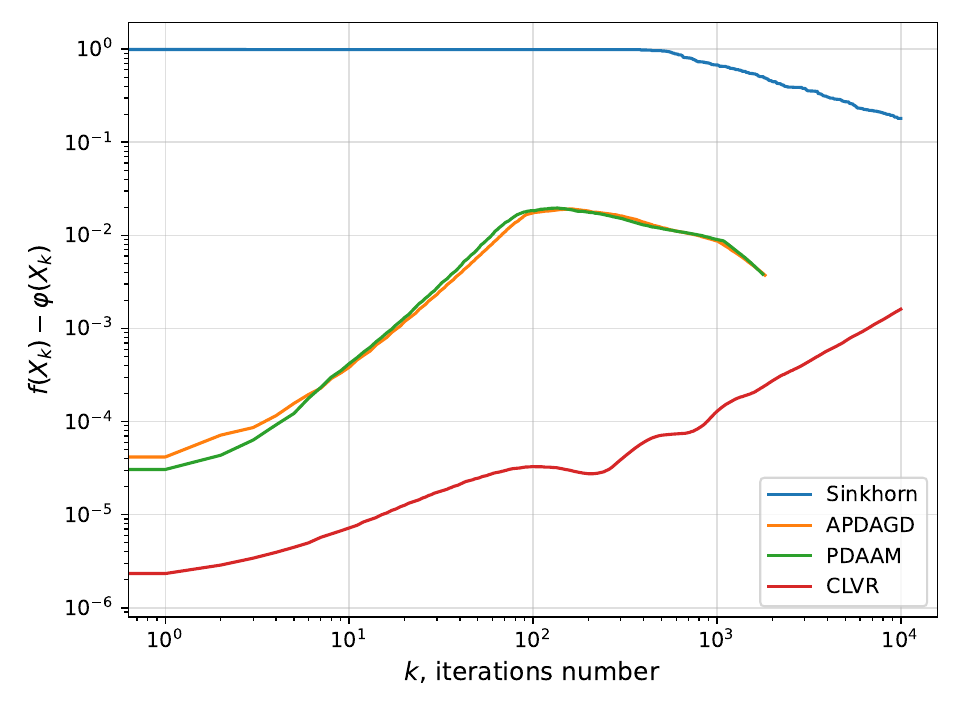}
         \caption{Dual gap convergence curves, $\varepsilon \propto 10^{-2}$.}
        \label{fig:new_gap_a}
     \end{subfigure}
     \begin{subfigure}[b]{0.3\textwidth}
         \centering
         \includegraphics[width=\textwidth]{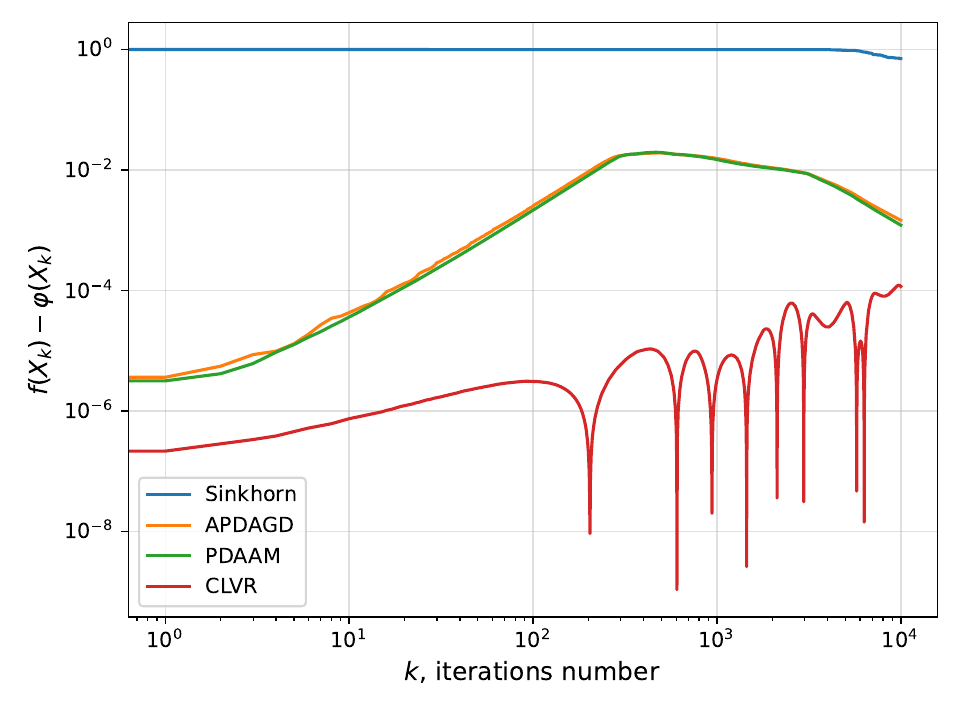}
         \caption{Dual gap convergence curves, $\varepsilon \propto 10^{-3}$.}
        \label{fig:new_gap_b}
     \end{subfigure}
     \begin{subfigure}[b]{0.3\textwidth}
         \centering
         \includegraphics[width=\textwidth]{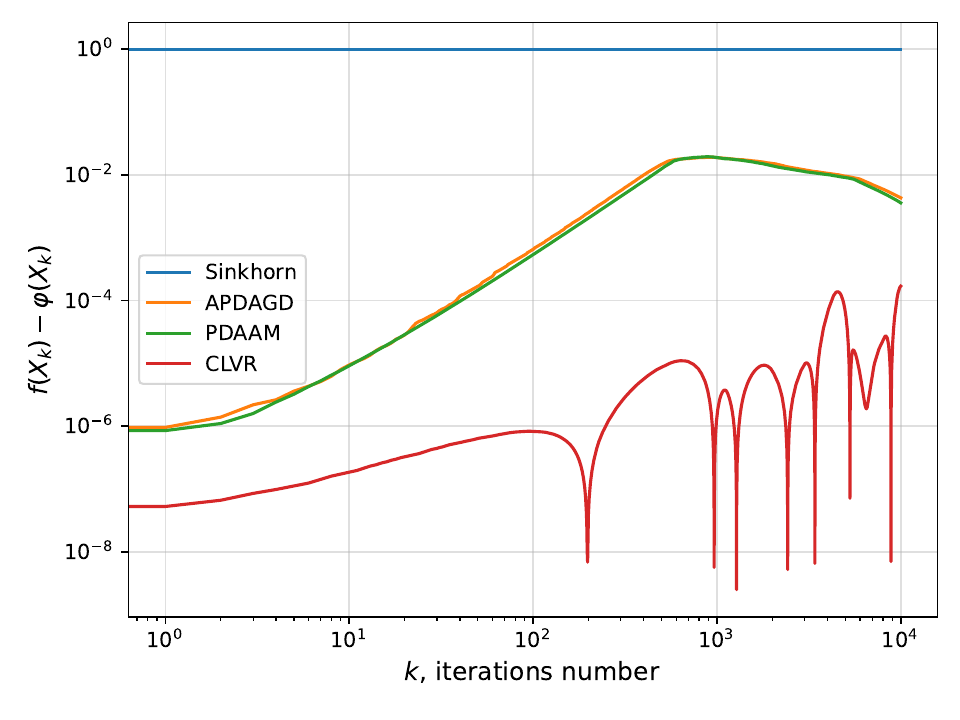}
         \caption{Dual gap convergence curves, $\varepsilon \propto 10^{-4}$.}
        \label{fig:new_gap_c}
     \end{subfigure}
        \caption{Practical efficiency of Sinkhorn--Knopp, Adaptive Accelerated Gradient, and Accelerated Alternating methods applied to Euclidean-regularised OT problem on MNIST dataset.}
        \label{fig:Euclidean}
\end{figure}

Secondly, the same algorithms were compared while applied to Euclidean-regularised OT problem. Figure~\ref{fig:Euclidean} shows convergence curves of methods, organisation of the plots is the same as above. One can see that ordering of the methods' performance remain the same as in the case of entropy-regularised OT. Specifically, the PDAAM algorithm convergence is faster than that of APDAGD and Sinkhorn. On the other hand, difference between PDAAM and APDAGD performance is less significant in the case of Euclidean-regularised OT (we conclude that progress of step which is optimal with respect to one of the dual variables is not much bigger than progress of the gradient step), and Sinkhorn algorithm performs significantly worse than in entropy-regularised OT and is not efficient in practice. CLVR did not displayed itself an efficient method in our experiment. Generally, convergence of all of the algorithms in the case of Euclidean regularisation is more prone to slowing down on the latter iterations.  

The expected property of Euclidean-regularised OT that the optimal transport plan obtained with it is sparse is approved in our experiments. One can see the examples of transport plans in Figure~\ref{fig:sparse}, the fraction of zero elements (which are $< 10^{-21}$) in them is around 99.5\%.

\section{Discussion} \label{sec:conclusion}
Euclidean regularisation for OT problems has been recently explored in several papers due to its practically valuable properties, such as robustness to small regularisation parameter and sparsity of the optimal transport plan. This paper provides a theoretical analysis of various algorithms that are applicable efficiently to Euclidean-regularised OT. We demonstrate and compare their practical performance. Our findings reveal that these desirable properties come at a cost. Namely, the slower convergence of all the algorithms and faster increase in arithmetic complexity as dimensionality grows.

Our plans involve considering different convex optimisation algorithms applied to Euclidean-regularised OT, focusing on splitting algorithms that are to be more computationally stable with small regularisation parameter \cite{lindback2023bringing}. Additionally, we aim to explore the application of Euclidean regularisation for the Wasserstein barycenter problem.
\begin{figure}[ht!]
     \centering
     \includegraphics[width=0.7\textwidth]{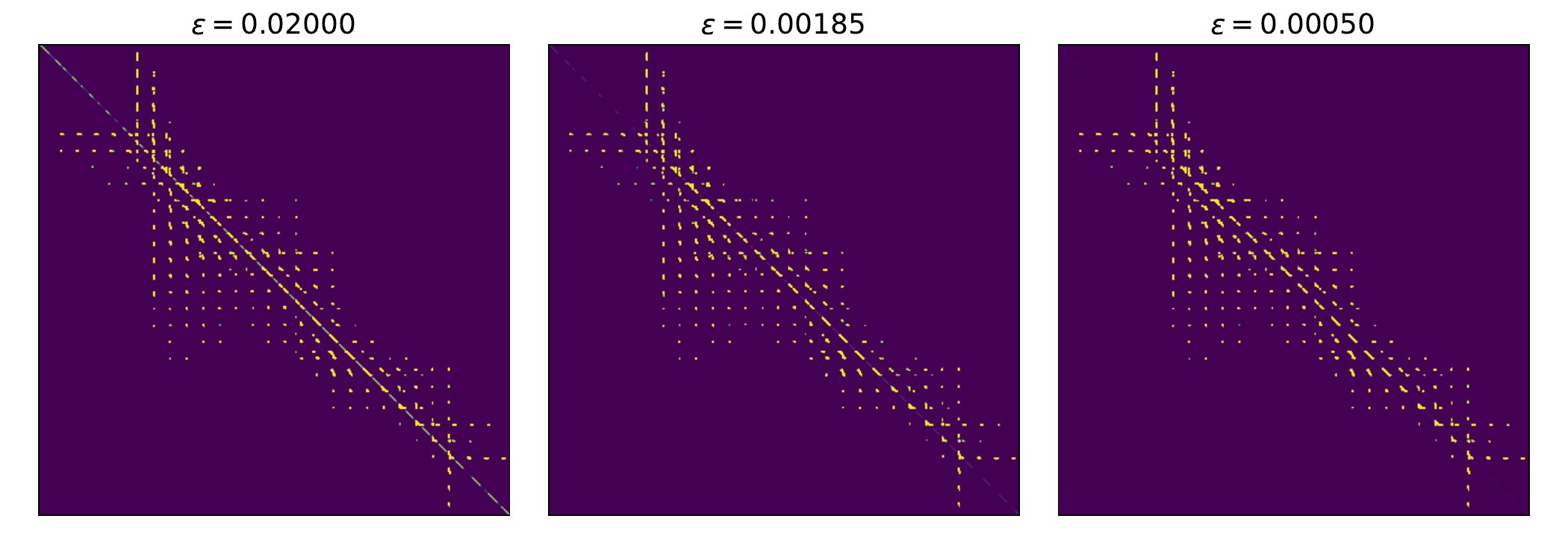}
     \caption{Sparse optimal transport plans obtained by Adaptive Accelerated Gradient Descent applied to Euclidean-regularised OT problem on MNIST dataset, 99.5\% of zero elements.}
        \label{fig:sparse}
\end{figure}

\bibliographystyle{plain}
\bibliography{main}

\end{document}